\documentclass[12pt]{article}
\usepackage{amssymb,oldlfont}
\usepackage{amsmath}
\usepackage{hyperref}
\usepackage{color}
\definecolor{myorange}{RGB}{180,90,0}
\definecolor{mygreen}{RGB}{70,140,0}
\usepackage[normalem]{ulem}

\usepackage[T1]{fontenc}
\usepackage[utf8]{inputenc}
\def\wrtext#1{\relax\ifmmode{\leavevmode\hbox{#1}}\else{#1}\fi}
\def\abs#1{\left|#1\right|}
\def\begeq{\begin{equation}}
\def\endeq{\end{equation}}

\def\part#1{\frac{\partial}{\partial #1}}

\def\norm#1{||\,#1\,||}

\newcommand{\real}{\mbox{\bf R}}
\newcommand{\comp}{\mbox{\bf C}}

\renewcommand{\exp}{\mbox{\rm exp\,}}

\newtheorem{dref}{Definition}[section]
\newtheorem{lemma}[dref]{Lemma}
\newtheorem{theo}[dref]{Theorem}
\newtheorem{prop}[dref]{Proposition}

\newenvironment{proof}{\vspace{.3cm}\noindent{{\em Proof:}}}{\hfill$\Box$}

\title{Characterizing boundedness of metaplectic Toeplitz operators}
\author{Lewis \textsc{Coburn} \footnote{Department of Mathematics, SUNY at Buffalo, Buffalo, NY 14260, USA, {\sf lcoburn@buffalo.edu}}\and Michael \textsc{Hitrik} \footnote{Department of Mathematics, UCLA, Los Angeles CA 90095-1555, USA, {\sf hitrik@math.ucla.edu}} \and \and Johannes \textsc{Sj\"ostrand}\footnote{IMB, Universit\'e de Bourgogne 9, Av. A. Savary, BP 47870
FR-21078 Dijon, France and UMR 5584 CNRS, {\sf johannes.sjostrand@u-bourgogne.fr}}}
\date{}

\begin{document}
\maketitle

\vspace*{1cm}
\noindent
{\bf Abstract}: We study Toeplitz operators on the Bargmann space, with Toeplitz symbols given by exponentials of complex quadratic forms. We show that the boundedness of the corresponding Weyl symbols is necessary for the boundedness of the operators, thereby completing the proof of the Berger-Coburn conjecture in this case. We also show that the compactness of such Toeplitz operators is equivalent to the vanishing of their Weyl symbols at infinity.

\tableofcontents
\section{Introduction and statement of results}
\setcounter{equation}{0}
\label{sec_introduction}
The Berger-Coburn conjecture~\cite{BC94},~\cite{LC19}, a long standing conjecture in the theory of Toeplitz operators, states that a Toeplitz operator on the Bargmann space is bounded precisely when its Weyl symbol is bounded. Important partial results towards the conjecture have been obtained in~\cite{BC94}. In our recent work~\cite{CoHiSj}, certain links have been established between the theory of Toeplitz operators on the Bargmann space and Fourier integral operators (FIOs) in the complex domain. The point of view of complex FIOs has been used in~\cite{CoHiSj} to show the sufficiency part of the conjecture, in the special case of Toeplitz symbols given by exponentials of complex quadratic forms. An extension of this result to the case of Toeplitz symbols that are exponentials of inhomogeneous quadratic polynomials has been obtained in the follow up paper~\cite{CoHiSjWh}, still relying on the FIO point of view. The necessity part of the Berger-Coburn conjecture for such metaplectic Toeplitz operators has been left open in~\cite{CoHiSj},~\cite{CoHiSjWh}, and it is the purpose of this note to finally settle it, showing that the boundedness of the Weyl symbols is necessary for the boundedness of the corresponding Toeplitz operators. We shall now proceed to describe the assumptions and state the main results of this work.

\bigskip
\noindent
Let $\Phi_0$ be a strictly plurisubharmonic quadratic form on $\comp^n$ and let us set
\begeq
\label{eq1.1}
\Lambda_{\Phi_0} = \left\{\left(x,\frac{2}{i}\frac{\partial \Phi_0}{\partial x}(x)\right), \, x\in \comp^n\right\} \subset \comp^n_x \times \comp^n_{\xi} = \comp^{2n}.
\endeq
The real linear subspace $\Lambda_{\Phi_0}$ is I-Lagrangian and R-symplectic, in the sense that the restriction of the complex symplectic form on $\comp^{2n}$ to $\Lambda_{\Phi_0}$ is real and non-degenerate. We may also recall that in the context of the Weyl quantization in the complex domain, the subspace $\Lambda_{\Phi_0}$ is to be regarded as the real phase space, see \cite{Sj95},~\cite{Zw12}.

\medskip
\noindent
Associated to the quadratic form $\Phi_0$ is the Bargmann space
\begeq
\label{eq1.2}
H_{\Phi_0}(\comp^n) = L^2(\comp^n, e^{-2\Phi_0} L(dx)) \cap {\rm Hol}(\comp^n),
\endeq
where $L(dx)$ is the Lebesgue measure on $\comp^n$. We have the orthogonal projection
\begeq
\label{eq1.3}
\Pi_{\Phi_0}: L^2(\comp^n,e^{-2\Phi_0} L(dx)) \rightarrow H_{\Phi_0}(\comp^n).
\endeq

\bigskip
\noindent
Let $q$ be a complex valued quadratic form on $\comp^n_x$, i.e. a homogenous polynomial of degree 2 in $x$, $\overline{x}$. We shall assume that
\begeq
\label{eq1.4}
{\rm Re}\, q(x) < \Phi_{\rm herm}(x) := (1/2)\left(\Phi_0(x) + \Phi_0(ix)\right),\quad x \neq 0.
\endeq
In this work, following \cite{CoHiSj}, we shall be concerned with (densely defined) Toeplitz operators of the form
\begeq
\label{eq1.5}
{\rm Top}(e^q) = \Pi_{\Phi_0} \circ e^{q} \circ \Pi_{\Phi_0}: H_{\Phi_0}(\comp^n) \rightarrow H_{\Phi_0}(\comp^n).
\endeq
Such operators can be represented using the Weyl quantization,
\begeq
\label{eq1.6}
{\rm Top}(e^q) = a^w(x,D_x),
\endeq
see~\cite{Sj95},~\cite{Zw12}, where the Weyl symbol $a\in C^{\infty}(\Lambda_{\Phi_0})$ is given by
\begeq
\label{eq1.7}
a\left(x,\frac{2}{i}\frac{\partial \Phi_0}{\partial x}(x)\right) = \left(\exp\left(\frac{1}{4} \left(\Phi''_{0,x\overline{x}}\right)^{-1} \partial_x \cdot \partial_{\overline{x}}\right)e^q\right)(x).
\endeq
The following is the first main result of this work.

\begin{theo}
\label{theo_main1}
Let $\Phi_0$ be a strictly plurisubharmonic quadratic form on $\comp^n$ and let $q$ be a complex valued quadratic form on $\comp^n$. Assume that
\begeq
\label{eq1.8}
{\rm Re}\, q(x) < \Phi_{\rm herm}(x) := (1/2)\left(\Phi_0(x) + \Phi_0(ix)\right),\quad x \neq 0
\endeq
and
\begeq
\label{eq1.9}
{\rm det}\, \partial_x \partial_{\overline{x}} \left(2\Phi_0 - q\right) \neq 0.
\endeq
The Toeplitz operator
$$
{\rm Top}(e^q): H_{\Phi_0}(\comp^n) \rightarrow H_{\Phi_0}(\comp^n))
$$
is bounded if and only if the Weyl symbol $a\in C^{\infty}(\Lambda_{\Phi_0})$ of ${\rm Top}(e^q)$ satisfies $a\in L^{\infty}(\Lambda_{\Phi_0})$.
\end{theo}

\medskip
\noindent
{\it Remark}. As mentioned above, the sufficiency of the boundedness of the Weyl symbol of ${\rm Top}(e^q)$ for the boundedness of the Toeplitz operator has been established in~\cite{CoHiSj}, and here we shall only be concerned with the proof of the necessity. In the special case of Toeplitz symbols of the form $e^q$, where $q$ is a quadratic form on $\comp^n$ satisfying (\ref{eq1.8}), (\ref{eq1.9}), Theorem \ref{theo_main1} establishes the validity of the Berger-Coburn conjecture~\cite{BC94},~\cite{LC19}.

\bigskip
\noindent
The compactness of Toeplitz operators of the form (\ref{eq1.5}) can also be characterized in terms of their Weyl symbols, in agreement with a  general conjecture made in~\cite{BaCoIs}.

\begin{theo}
\label{theo_main2}
Let $\Phi_0$ be a strictly plurisubharmonic quadratic form on $\comp^n$ and let $q$ be a complex valued quadratic form on $\comp^n$, satisfying {\rm (\ref{eq1.8})}, {\rm (\ref{eq1.9})}. The Toeplitz operator
$$
{\rm Top}(e^q): H_{\Phi_0}(\comp^n) \rightarrow H_{\Phi_0}(\comp^n)
$$
is compact if and only if the Weyl symbol $a\in C^{\infty}(\Lambda_{\Phi_0})$ of ${\rm Top}(e^q)$ vanishes at infinity.
\end{theo}

\medskip
\noindent
{\it Remark}. While the Toeplitz symbols that we consider here, given by exponentials of complex quadratic forms, form a very restrictive class, they may still be of some interest since the class of the associated Toeplitz operators includes those that are "at the edge" of boundedness, with the unboundedness of the symbols attenuated by their rapid oscillations at infinity. See also~\cite{BC94},~\cite{CoHiSj}.

\medskip
\noindent
The plan of the note is as follows. In Section \ref{sec_boundedness}, we prove Theorem \ref{theo_main1}, and Theorem \ref{theo_main2} is then established in Section \ref{sec_compactness}. Following~\cite{CoHiSjWh}, the principal idea here is to let the (bounded) Toeplitz operator act on the space of normalized coherent states on the Bargmann space. It turns out that this allows one to extract directly the boundedness of the Weyl symbol, or rather the positivity of the corresponding complex linear canonical transformation, all the while relying crucially on the techniques developed in~\cite{CoHiSj}. Section \ref{sec_example} is devoted to the discussion of an explicit family of metaplectic Toeplitz operators on a model Bargmann space, illustrating Theorem \ref{theo_main1} and Theorem \ref{theo_main2}.

\medskip
\noindent
{\bf Acknowledgments}. We are very grateful to Joe Viola and Francis White for helpful discussions.

\medskip
\noindent
{\bf In Memoriam.}  Charles Arnold Berger (1937--2014) was professor of mathematics at CUNY.  Challenged by Crohn's Disease, he was tenacious and insightful in mathema\-ti\-cal research, with flashes of brilliance and humor.

\section{From bounded Toeplitz operators to bounded Weyl symbols: Proof of Theorem 1.1}
\label{sec_boundedness}
\setcounter{equation}{0}
Let $\Phi_0$ be a strictly plurisubharmonic quadratic form on $\comp^n$ and let $q$ be a complex valued quadratic form on $\comp^n$ satisfying (\ref{eq1.8}), (\ref{eq1.9}). We recall from~\cite{CoHiSj} that when equipped with the natural (maximal) domain
\begeq
\label{eq2.1}
{\cal D}({\rm Top}(e^q)) = \left\{u\in H_{\Phi_0}(\comp^n); e^{q}\,u \in L^2(\comp^n, e^{-2\Phi_0}L(dx))\right\},
\endeq
the Toeplitz operator
\begeq
\label{eq2.2}
{\rm Top}(e^{q}) = \Pi_{\Phi_0} \circ e^{q} \circ \Pi_{\Phi_0}: H_{\Phi_0}(\comp^n) \rightarrow H_{\Phi_0}(\comp^n)
\endeq
becomes densely defined.

\medskip
\noindent
Let $a\in C^{\infty}(\Lambda_{\Phi_0})$ be the Weyl symbol of ${\rm Top}(e^q)$, introduced in (\ref{eq1.7}) and let us recall that the implication $a\in L^{\infty}(\Lambda_{\Phi_0}) \Longrightarrow {\rm Top}(e^q)\in {\cal L}(H_{\Phi_0}(\comp^n),H_{\Phi_0}(\comp^n))$ has already been established in~\cite[Theorem 1.2]{CoHiSj}. We only need therefore to check the necessity of the condition $a\in L^{\infty}(\Lambda_{\Phi_0})$ for the boundedness of the Toeplitz operator.

\medskip
\noindent
When doing so, following~\cite{CoHiSj},~\cite{CoHiSjWh}, let us write for
$u\in {\cal D}({\rm Top}(e^q))$,
\begeq
\label{eq2.3}
{\rm Top}(e^{q})u(x) = C \int\!\!\!\int_{\Gamma} e^{2(\Psi_0(x,\theta) - \Psi_0(y,\theta)) + Q(y,\theta)} u(y)\, dy\, d\theta, \quad C\neq 0.
\endeq
Here $\Gamma$ is the contour in $\comp^{2n}_{y,\theta}$, given by $\theta = \overline{y}$, and $\Psi_0$ and $Q$ are the polarizations of $\Phi_0$ and $q$, respectively, i.e. holomorphic quadratic forms on $\comp^{2n}_{y,\theta}$ such that $\Psi_0|_{\Gamma} = \Phi_0$, $Q|_{\Gamma} = q$. Following~\cite{CoHiSj}, we shall view ${\rm Top}(e^q)$ in (\ref{eq2.3}) as a metaplectic Fourier integral operator in the complex domain. The holomorphic quadratic form
\begeq
\label{eq2.4}
F(x,y,\theta) = \frac{2}{i}\left(\Psi_0(x,\theta) - \Psi_0(y,\theta)\right) + \frac{1}{i}Q(y,\theta)
\endeq
is a non-degenerate phase function in the sense of H\"ormander~\cite{H_FIOI}, in view of the non-degeneracy of $\partial_x \partial_{\theta} \Psi_0$, and using also (\ref{eq1.9}) we conclude as in~\cite{CoHiSj}, that the associated canonical relation
\begeq
\label{eq2.5}
\kappa: \comp^{2n} \ni \left(y,-F'_y(x,y,\theta)\right) \mapsto \left(x,F'_x(x,y,\theta)\right)\in \comp^{2n},\quad F'_{\theta}(x,y,\theta) = 0,
\endeq
is the graph of a complex linear canonical transformation. It follows from~\cite[Proposition B.1]{CoHiSj} that the Weyl symbol $a$ of ${\rm Top}(e^q)$ satisfies $a\in L^{\infty}(\Lambda_{\Phi_0})$  precisely when the canonical transformation $\kappa$ in (\ref{eq2.5}) is positive relative to $\Lambda_{\Phi_0}$, i.e.,
\begeq
\label{eq2.6}
\frac{1}{i} \biggl(\sigma(\kappa(\rho), \iota_{\Phi_0} \kappa(\rho)) - \sigma(\rho, \iota_{\Phi_0}(\rho))\biggr) \geq 0,\quad \rho \in \comp^{2n}.
\endeq
Here $\iota_{\Phi_0}: \comp^{2n} \rightarrow \comp^{2n}$ is the unique anti-linear involution such that $\iota|_{\Lambda_{\Phi_0}} = 1$, see~\cite{CoHiSj}, and
\begeq
\label{eq2.7}
\sigma = \sum_{j=1}^n d\xi_j \wedge dx_j
\endeq
is the complex symplectic form on $\comp^{2n} = \comp^n_x \times \comp^n_{\xi}$.

\bigskip
\noindent
Assuming that
\begeq
\label{eq2.71}
{\rm Top}(e^q)\in {\cal L}(H_{\Phi_0}(\comp^n),H_{\Phi_0}(\comp^n)),
\endeq
let us check first that it suffices to show that $a\in L^{\infty}(\Lambda_{\Phi_0})$, when the pluriharmonic part of $\Phi_0$ vanishes. Indeed, let us decompose
\begeq
\label{eq2.8}
\Phi_0 = \Phi_{{\rm herm}} + \Phi_{{\rm plh}},
\endeq
where $\Phi_{{\rm herm}}(x) = (\Phi_0)''_{\overline{x}\, x}x\cdot \overline{x}$ is Hermitian positive definite, and
$\Phi_{{\rm plh}}(x) = {\rm Re} f(x)$, with $f(x) = (\Phi_0)''_{xx}x\cdot x$, is pluriharmonic. Let
\begeq
\label{eq2.9}
A = \frac{2}{i}\left(\Phi_0\right)''_{xx},
\endeq
and following~\cite{CoHiSj},~\cite{HiLaSjZe}, let us introduce the complex linear canonical transformation
\begeq
\label{eq2.10}
\kappa_A: \comp^{2n}\ni (y,\eta) \mapsto (y, \eta - Ay)\in \comp^{2n},
\endeq
satisfying
\begeq
\label{eq2.11}
\kappa_A(\Lambda_{\Phi_0}) = \Lambda_{\Phi_{{\rm herm}}}.
\endeq
Associated to $\kappa_A$ is the unitary metaplectic Fourier integral operator
\begeq
\label{eq2.12}
{\cal U}: H_{\Phi_0}(\comp^n) \ni u \mapsto u e^{-f}\in H_{\Phi_{\rm herm}}(\comp^n),
\endeq
and letting
\begeq
\label{eq2.13}
\Pi_{\Phi_{{\rm herm}}}: L^2(\comp^n, e^{-2\Phi_{{\rm herm}}}L(dx)) \rightarrow H_{\Phi_{{\rm herm}}}(\comp^n)
\endeq
be the orthogonal projection, we observe that
\begeq
\label{eq2.14}
\Pi_{\Phi_0} = {\cal U}^{-1} \circ \Pi_{\Phi_{{\rm herm}}} \circ {\cal U}.
\endeq
The assumption (\ref{eq2.71}) implies therefore that the Toeplitz operator
\begeq
\label{eq2.15}
\Pi_{\Phi_{{\rm herm}}} \circ e^{q} \circ  \Pi_{\Phi_{{\rm herm}}} = {\cal U} \circ {\rm Top}(e^q) \circ {\cal U}^{-1}:
H_{\Phi_{{\rm herm}}}(\comp^n) \rightarrow H_{\Phi_{{\rm herm}}}(\comp^n)
\endeq
is bounded, and arguing as above, we may regard the operator (\ref{eq2.15}) as a metaplectic Fourier integral operator associated to the complex linear canonical transformation
\begeq
\label{eq2.16}
\kappa_{{\rm herm}}: (y, -\partial_y F_{{\rm herm}}(x,y,\theta)) \mapsto (x, \partial_x F_{{\rm herm}}(x,y,\theta)),\quad \partial_{\theta} F_{{\rm herm}}(x,y,\theta) = 0.
\endeq
Here
\begeq
\label{eq2.17}
F_{{\rm herm}}(x,y,\theta) = \frac{2}{i}\left(\Psi_{{\rm herm}}(x,\theta) - \Psi_{{\rm herm}}(y,\theta)\right) + \frac{1}{i}Q(y,\theta),
\endeq
with $\Psi_{{\rm herm}}$ being the polarization of $\Phi_{{\rm herm}}$. Observing that
$$
F(x,y,\theta) = F_{{\rm herm}}(x,y,\theta) + \frac{1}{2}Ax\cdot x - \frac{1}{2}Ay\cdot y,
$$
we see that the canonical transformation $\kappa$ in (\ref{eq2.5}) admits the factorization
\begeq
\label{eq2.18}
\kappa = \kappa_A^{-1} \circ \kappa_{{\rm herm}} \circ \kappa_A.
\endeq
Combining (\ref{eq2.18}) with the fact that
\begeq
\label{eq2.19}
\iota_{\Phi_0} = \kappa_A^{-1} \circ \iota_{\Phi_{\rm herm}} \circ \kappa_A,
\endeq
see~\cite{CoHiSj}, we conclude that $\kappa$ is positive relative to $\Lambda_{\Phi_0}$, i.e. that (\ref{eq2.6}) holds, precisely when the canonical transformation $\kappa_{{\rm herm}}$ is positive relative to $\Lambda_{\Phi_{{\rm herm}}}$, i.e.,
\begeq
\label{eq2.20}
\frac{1}{i}\bigl(\sigma(\kappa_{{\rm herm}}(\rho),\iota_{\Phi_{\rm herm}}\kappa_{{\rm herm}}(\rho)) - \sigma(\rho, \iota_{\Phi_{\rm herm}}(\rho))\bigr)\geq 0,\quad \rho \in \comp^{2n}.
\endeq
In what follows, we shall assume therefore that the pluriharmonic part of $\Phi_0$ vanishes, so that
\begeq
\label{eq2.21}
\Phi_0(x) = (\Phi_0)''_{\overline{x}x}x\cdot \overline{x},\quad x\in \comp^n,
\endeq
and
\begeq
\label{eq2.22}
\Psi_0(x,y) = (\Phi_0)''_{\overline{x}x}x\cdot y,\quad x,y\in \comp^n.
\endeq

\bigskip
\noindent
We have
\begeq
\label{eq2.23}
2{\rm Re}\, \Psi_0(x,\overline{y}) - \Phi_0(x) - \Phi_0(y) = -(\Phi''_0)_{\overline{x}x}(x-y)\cdot (\overline{x-y}) = -\Phi_0(x-y)
\endeq
Assuming that (\ref{eq2.71}) holds, following~\cite{CoHiSjWh}, we shall examine the action of ${\rm Top}(e^q)$ on the space of "coherent states", i.e. the normalized reproducing kernels for the Bargmann space $H_{\Phi_0}(\comp^n)$. Let us set
\begeq
\label{eq2.24}
k_w(x) =  C_{\Phi_0}\, e^{2\Psi_0(x,\overline{w}) - \Phi_0(w)},\quad w\in \comp^n,
\endeq
Using (\ref{eq2.23}) and recalling (\ref{eq1.8}), we see that
\begeq
\label{eq2.24.1}
k_w \in {\cal D}({\rm Top}(e^q)), \quad w\in \comp^n,
\endeq
and choosing the constant $C_{\Phi_0} > 0$ suitably, we achieve that $\norm{k_w}_{H_{\Phi_0}({\bf C}^n)} = 1$, $w\in \comp^n$. We may write therefore, in view of (\ref{eq2.3}),
\begeq
\label{eq2.25}
\left({\rm Top}(e^{q})k_w\right)(x) = C\,C_{\Phi_0}\, e^{-\Phi_0(w)} \int\!\!\!\int_{\Gamma} e^{2\Psi_0(x,\theta) + Q(y,\theta) + 2\Psi_0(y,\overline{w}) - 2\Psi_0(y,\theta)} \,dy\,d\theta.
\endeq
Let us next make the following general observation.

\begin{prop}
\label{prop_polar}
Let $g$ be a complex valued quadratic form on $\comp^n$ such that ${\rm Re}\, g < 0$ in the sense of quadratic forms. Let $G$ be the polarization of $g$. Then the holomorphic quadratic form $G$ on $\comp^{2n}$ is non-degenerate.
\end{prop}
\begin{proof}
The pluriharmonic quadratic form ${\rm Re}\, G$ satisfies
\begeq
\label{eq2.26}
{\rm Re}\, G|_{\Gamma} < 0,
\endeq
where $\Gamma \subset \comp^{2n}_{x,y}$ is the anti-diagonal, $y = \overline{x}$. It follows that the signature of ${\rm Re}\, G$ is $(2n, 2n)$ so that
${\rm Re}\, G$ is non-degenerate,
\begeq
\label{eq2.27}
\abs{\nabla {\rm Re}\, G(x,y)} \asymp \abs{x} + \abs{y} \quad \wrtext{on}\,\,\, \comp^{2n}_{x,y},
\endeq
where the gradient is taken in the real sense of $\real^{4n}$. The result follows in view of the following general fact: let $U\in {\rm Hol}(\comp^N_z)$. Then
\begeq
\label{2.28}
\abs{\nabla {\rm Re}\, U(z)} = 2\abs{\partial_z {\rm Re}\, U(z)} = \abs{\partial_z U(z)}.
\endeq
\end{proof}

\medskip
\noindent
An application of Proposition \ref{prop_polar} together with (\ref{eq1.8}) allows us to conclude that the holomorphic quadratic form
\begeq
\label{eq2.29}
\comp^{2n}_{y,\theta} \ni (y,\theta) \mapsto Q(y,\theta) - 2\Psi_0(y,\theta)
\endeq
is non-degenerate, and an application of the method of exact (quadratic) stationary phase~\cite[Lemma 13.2]{Zw12} to (\ref{eq2.25}) gives therefore, with a new constant $C$,
\begeq
\label{eq2.30}
\left({\rm Top}(e^{q})k_w\right)(x) = C\, e^{2f(x,\overline{w})- \Phi_0(w)}, \quad 0\neq C\in \comp.
\endeq
Here $f(x,z)$ is a holomorphic quadratic form on $\comp^{2n}_{x,z}$ given by
\begeq
\label{eq2.31}
2 f(x,z) = {\rm vc}_{y,\theta}\left(2 \Psi_0(x,\theta) + Q(y,\theta) + 2 \Psi_0(y,z) - 2\Psi_0(y,\theta)\right).
\endeq
Here we write "${\rm vc}$" for the critical value.

\medskip
\noindent
For future reference, let us make the following observation.
\begin{prop}
\label{prop_det}
We have
\begeq
\label{eq2.32}
{\rm det}\, f''_{xz} \neq 0.
\endeq
\end{prop}
\begin{proof}
We have, in view of (\ref{eq2.31}) and (\ref{eq2.22}),
\begeq
\label{eq2.33}
2 f(x,z) = {\rm vc}_{y,\theta}\left(2(\Phi_0)''_{\overline{x}x}x\cdot \theta + Q(y,\theta) + 2(\Phi_0)''_{\overline{x}x}y\cdot z
- 2(\Phi_0)''_{\overline{x}x}y\cdot \theta\right),
\endeq
and letting $(y,\theta) = (y(x,z),\theta(x,z))\in \comp^n \times \comp^n$ be the unique critical point corresponding to the critical value in (\ref{eq2.33}), we see that
\begeq
\label{eq2.34}
f'_x(x,z) = (\partial_x \Psi_0)(x,\theta) = (\Phi_0)''_{x\overline{x}}\,\theta(x,z),\quad f''_{xz} = (\Phi_0)''_{x\overline{x}}\,
\partial_z \theta(x,z).
\endeq
We have to show that $\partial_z \theta(x,z)$ is invertible, and to this end we observe that the critical point $(y,\theta) = (y(x,z),\theta(x,z))$ satisfies
\begeq
\label{eq2.35}
2(\Phi_0)''_{\overline{x}x} y - Q'_{\theta}(y,\theta) = 2(\Phi_0)''_{\overline{x}x} x, \quad 2(\Phi_0)''_{x\overline{x}}\theta - Q'_y(y,\theta) =
2(\Phi_0)''_{x\overline{x}}z.
\endeq
Writing
\begeq
\label{eq2.36}
Q(y,\theta) = \frac{1}{2} Q''_{yy}y\cdot y + Q''_{y\theta}\theta\cdot y + \frac{1}{2} Q''_{\theta\,\theta}\theta \cdot \theta, \quad (y,\theta) \in \comp^n \times\comp^n,
\endeq
we see that the equations (\ref{eq2.35}) take the form,
\begeq
\label{eq2.37}
\begin{pmatrix}
A_{11} & A_{12} \\\
A_{21} & A_{22}
\end{pmatrix} \begin{pmatrix}
y \\\
\theta
\end{pmatrix} = \begin{pmatrix}
2(\Phi_0)''_{\overline{x}x}x \\\
2(\Phi_0)''_{x\overline{x}}z
\end{pmatrix}.
\endeq
Here the $2n\times 2n$ matrix
\begeq
\label{eq2.38}
{\cal A} = \begin{pmatrix}
A_{11} & A_{12} \\\
A_{21} & A_{22}
\end{pmatrix} = \begin{pmatrix}
2(\Phi_0)''_{\overline{x}x} - Q''_{\theta y} & -Q''_{\theta\,\theta} \\\
-Q''_{yy} & 2(\Phi_0)''_{x\overline{x}}- Q''_{y \theta}
\end{pmatrix}.
\endeq
is invertible, in view of the non-degeneracy of the quadratic form in (\ref{eq2.29}), and furthermore, $A_{11}$ is invertible, thanks to (\ref{eq1.9}). Letting
\begeq
\label{eq2.39}
{\cal B} = {\cal A}^{-1} = \begin{pmatrix}
B_{11} & B_{12} \\\
B_{21} & B_{22}
\end{pmatrix},
\endeq
we conclude that $B_{22}$ is invertible, in view of the Schur complement formula, see~\cite[Lemma 3.1]{SjZw}. It follows that $\partial_z \theta(x,z)$  is invertible, and (\ref{eq2.32}) follows, in view of (\ref{eq2.34}).
\end{proof}

\bigskip
\noindent
It follows from (\ref{eq2.24.1}) and (\ref{eq2.30}) that
\begeq
\label{eq2.40}
e^{2f(\cdot, \overline{w})} \in H_{\Phi_0}(\comp^n), \quad w\in \comp^n,
\endeq
and in particular, we infer from (\ref{eq2.40}) that
\begeq
\label{eq2.41}
2{\rm Re}\, f(x,0) - \Phi_0(x) < 0,\quad 0\neq x\in \comp^n.
\endeq
Using (\ref{eq2.30}) and writing
\begeq
\label{eq2.42}
\norm{{\rm Top}(e^q) k_w}^2_{H_{\Phi_0}({\bf C}^n)} = C^2 e^{-2\Phi_0(w)} \int e^{4{\rm Re}\, f(x,\overline{w}) - 2\Phi_0(x)}\,L(dx),
\endeq
we conclude, in view of (\ref{eq2.41}) and the quadratic version of stationary phase (the Laplace method)~\cite[Lemma 13.2]{Zw12} that
\begeq
\label{eq2.43}
\norm{{\rm Top}(e^q) k_w}^2_{H_{\Phi_0}({\bf C}^n)} = \widetilde{C}^2 e^{-2\Phi_0(w)} \exp\left({\rm sup}_x\left(4{\rm Re}\, f(x,\overline{w}) - 2\Phi_0(x)\right)\right),\quad \widetilde{C}\neq 0.
\endeq
We get therefore the following necessary condition for the boundedness of the Toeplitz operator ${\rm Top}(e^q)$ on $H_{\Phi_0}(\comp^n)$,
\begeq
\label{eq2.44}
{\rm sup}_x\left(4{\rm Re}\, f(x,\overline{w}) - 2\Phi_0(x)\right) - 2\Phi_0(w) \leq 0, \quad w\in \comp^n,
\endeq
or in other words,
\begeq
\label{eq2.45}
2{\rm Re}\, f(x,\overline{w}) \leq \Phi_0(x) + \Phi_0(w), \quad (x,w)\in \comp^n_x \times \comp^n_w.
\endeq

\medskip
\noindent
Theorem \ref{theo_main1} follows therefore from the following result.

\begin{prop}
\label{prop_positive}
Assume that the condition {\rm (\ref{eq2.45})} holds, where the holomorphic quad\-ra\-tic form $f(x,z)$ is given in {\rm (\ref{eq2.31})}. Then the Weyl symbol $a\in C^{\infty}(\Lambda_{\Phi_0})$ of the Toeplitz operator ${\rm Top}(e^q)$ satisfies $a\in L^{\infty}(\Lambda_{\Phi_0})$.
\end{prop}
\begin{proof}
We shall prove that the canonical transformation $\kappa$ in (\ref{eq2.5}) is positive relative to $\Lambda_{\Phi_0}$. To this end, let us set
\begeq
\label{eq2.46}
\varphi(x,y,z) = \frac{2}{i} \left(f(x,z) - \Psi_0(y,z)\right).
\endeq
The phase function $\varphi(x,y,z)$ is non-degenerate in the sense of H\"ormander, with $z\in \comp^n$ viewed as the fiber variables, and it follows from Proposition \ref{prop_det} that the canonical relation
\begeq
\label{eq2.47}
\widetilde{\kappa}: \comp^{2n} \ni \left(y,-\varphi'_y(x,y,z)\right) \mapsto \left(x,\varphi'_x(x,y,z)\right)\in \comp^{2n},\quad \varphi'_{z}(x,y,z) = 0,
\endeq
is the graph of a complex linear canonical transformation, see~\cite{CGHS}. We have more explicitly,
\begeq
\label{eq2.48}
\widetilde{\kappa}: \left(y, \frac{2}{i}\partial_y \Psi_0(y,z)\right) \mapsto \left(x,\frac{2}{i} \partial_x f(x,z)\right),\quad \partial_z f(x,z) = \partial_z \Psi_0(y,z),
\endeq
or equivalently, recalling (\ref{eq2.22}),
\begeq
\label{eq2.49}
\widetilde{\kappa}: \left(y, \frac{2}{i} (\Phi_0)''_{x\overline{x}}z\right)\mapsto \left(x,\frac{2}{i} f'_x(x,z)\right),\quad f'_z(x,z) =
(\Phi_0)''_{\overline{x}x}y.
\endeq
In the proof of~\cite[Proposition 3.2]{CoHiSj} it is explained how the condition (\ref{eq2.45}) implies that the canonical transformation $\widetilde{\kappa}$ is positive relative to $\Lambda_{\Phi_0}$, and we claim now that in fact, $\widetilde{\kappa} = \kappa$ in (\ref{eq2.5}). Indeed, let us recall from~\cite{HiSj15},~\cite{Sj95} that the orthogonal projection $\Pi_{\Phi_0}$ in (\ref{eq1.3}) is given by
\begeq
\label{eq2.49.1}
\Pi_{\Phi_0}u(x) = a_0 \int\!\!\!\int e^{2\Psi_0(x,\overline{y})} u(y)\, e^{-2\Phi_0(y)}\,dy\,d\overline{y}\, \quad a_0 \neq 0,
\endeq
and applying ${\rm Top}(e^q)$ to
\begeq
\label{eq2.49.2}
u(x) = \Pi_{\Phi_0}u(x) = a_0 \int\!\!\!\int e^{2\Psi_0(x,\overline{y})} u(y)\, e^{-2\Phi_0(y)}\,dy\,d\overline{y}, \quad u\in {\cal D}({\rm Top}(e^q)),
\endeq
we get recalling (\ref{eq2.24}), (\ref{eq2.30}),
\begin{multline}
\label{eq2.49.3}
{\rm Top}(e^q)u(x) = a_0 \int\!\!\!\int ({\rm Top}(e^q)k_y)(x)\, u(y)\, e^{-\Phi_0(y)}\, dy\,d\overline{y} \\
= C \int\!\!\!\int e^{2f(x,\overline{y})} u(y)  e^{-2\Phi_0(y)}\, dy\,d\overline{y} = C \int\!\!\!\int_{\Gamma} e^{2(f(x,z)- \Psi_0(y,z))} u(y)\, dy\, dz, \quad C\neq 0.
\end{multline}
The representation (\ref{eq2.49.3}) gives another way of expressing the Fourier integral operator ${\rm Top}(e^q)$ in (\ref{eq2.3}), using the phase function in (\ref{eq2.46}) --- in the terminology of~\cite{CoHiSj}, (\ref{eq2.49.3}) gives the Bergman form for the operator
${\rm Top}(e^q)\in {\cal L}(H_{\Phi_0}(\comp^n), H_{\Phi_0}(\comp^n))$. We expect therefore the canonical transformations (\ref{eq2.47}) and (\ref{eq2.5}) to be equal, and let us also verify this fact by a direct computation. Using (\ref{eq2.4}), (\ref{eq2.5}), and (\ref{eq2.22}), we see that the canonical transformation $\kappa$ is of the form
\begeq
\label{eq2.50}
\kappa: \left(y, \frac{2}{i} (\Phi_0)''_{x\overline{x}}\theta - \frac{1}{i}Q'_y(y,\theta)\right) \mapsto
\left(x, \frac{2}{i} (\Phi_0)''_{x\overline{x}}\theta\right), \quad 2(\Phi_0)''_{\overline{x}x}(x-y) + Q'_{\theta}(y,\theta) = 0,
\endeq
or in other words,
\begeq
\label{eq2.51}
\kappa: \left(y, \frac{2}{i} (\Phi_0)''_{x\overline{x}}\theta - \frac{1}{i}Q'_y(y,\theta)\right) \mapsto
\left(y - \frac{1}{2} \left((\Phi_0)''_{\overline{x}x}\right)^{-1}Q'_{\theta}(y,\theta), \frac{2}{i} (\Phi_0)''_{x\overline{x}}\theta\right).
\endeq
On the other hand, writing in view of (\ref{eq2.33}),
\begeq
\label{eq2.52}
2 f(x,z) = {\rm vc}_{\widetilde{y},\theta}\left(2(\Phi_0)''_{\overline{x}x}x\cdot \theta + Q(\widetilde{y},\theta) + 2(\Phi_0)''_{\overline{x}x}\widetilde{y}\cdot z - 2(\Phi_0)''_{\overline{x}x}\widetilde{y}\cdot \theta\right),
\endeq
we obtain that
\begeq
\label{eq2.53}
f'_x(x,z) = (\Phi_0)''_{x\overline{x}}\theta(x,z),\quad f'_z(x,z) = (\Phi_0)''_{\overline{x}x}\widetilde{y}(x,z),
\endeq
with $(\widetilde{y}(x,z),\theta(x,z))$ being the unique critical point corresponding to the critical value in (\ref{eq2.52}). We get therefore, using (\ref{eq2.49}) and (\ref{eq2.53}),
\begeq
\label{eq2.54}
\widetilde{\kappa}: \left(y, \frac{2}{i} (\Phi_0)''_{x\overline{x}}z\right) \mapsto
\left(x, \frac{2}{i}(\Phi_0)''_{x\overline{x}}\theta(x,z)\right), \quad \widetilde{y}(x,z) = y.
\endeq
Here, as we have already seen in (\ref{eq2.35}), the critical point $(\widetilde{y},\theta) = (\widetilde{y}(x,z),\theta(x,z))$ satisfies
\begeq
\label{eq2.55}
2(\Phi_0)''_{\overline{x}x}x = 2(\Phi_0)''_{\overline{x}x}\widetilde{y} - Q'_{\theta}(\widetilde{y},\theta),\quad
2(\Phi_0)''_{x\overline{x}}z = 2(\Phi_0)''_{x\overline{x}}\theta - Q'_{y}(\widetilde{y},\theta).
\endeq
Comparing (\ref{eq2.54}), (\ref{eq2.55}) with (\ref{eq2.51}), we conclude that $\widetilde{\kappa} = \kappa$, and therefore, the latter canonical transformation is positive relative to $\Lambda_{\Phi_0}$. The proof is complete.
\end{proof}

\section{Characterizing compact Toeplitz operators: Proof of Theorem 1.2}
\label{sec_compactness}
\setcounter{equation}{0}
In this section, we let $q$ be a complex valued quadratic form on $\comp^n$ satisfying (\ref{eq1.8}), (\ref{eq1.9}), for a given strictly plurisubharmonic quadratic form $\Phi_0$ on $\comp^n$. Our purpose here is to establish Theorem \ref{theo_main2}, and when doing so we shall first verify that the vanishing of the Weyl symbol at infinity is a sufficient condition for the compactness of the Toeplitz operator ${\rm Top}(e^q)$ on $H_{\Phi_0}(\comp^n)$. Indeed, this result has essentially been established in~\cite{CoHiSj}.

\medskip
\noindent
Recalling that the Weyl symbol $a$ of ${\rm Top}(e^q)$ is given by (\ref{eq1.7}), we observe, following~\cite{CoHiSj} that we can write
\begeq
\label{eq3.1}
a(x,\xi) = C\, e^{iF(x,\xi)},\quad (x,\xi)\in \Lambda_{\Phi_0},
\endeq
for some $C\neq 0$, where $F$ is a holomorphic quadratic form on $\comp^{2n}_{x,\xi}$. It follows that the vanishing of $a$ it infinity is equivalent to the ellipticity property
\begeq
\label{eq3.2}
{\rm Im}\, F\left(x,\frac{2}{i}\frac{\partial \Phi_0}{\partial x}(x)\right) \asymp \abs{x}^2,\quad x\in \comp^n,
\endeq
and an application of~\cite[Proposition B.1]{CoHiSj} shows that (\ref{eq3.2}) is equivalent to the fact that the canonical transformation $\kappa$ in (\ref{eq2.5}) is strictly positive relative to $\Lambda_{\Phi_0}$, so that
\begeq
\label{eq3.3}
\frac{1}{i} \biggl(\sigma(\kappa(\rho), \iota_{\Phi_0} \kappa(\rho)) - \sigma(\rho, \iota_{\Phi_0}(\rho))\biggr) > 0,\quad 0\neq \rho \in \comp^{2n}.
\endeq
An application of~\cite[Proposition 3.3]{CoHiSj} gives that the operator ${\rm Top}(e^q)$ is compact, and in fact of trace class, on $H_{\Phi_0}(\comp^n)$, with rapidly decaying singular values.

\medskip
\noindent
When proving the necessity of the vanishing of the Weyl symbol at infinity for the compactness of ${\rm Top}(e^q)$, we may equivalently establish the strict positivity of the canonical transformation $\kappa$ in (\ref{eq2.5}), and arguing as in Section \ref{sec_boundedness}, we may first reduce to the case when the pluriharmonic part of $\Phi_0$ vanishes. Proceeding next as in Section \ref{sec_boundedness}, we shall consider the action of the compact operator ${\rm Top}(e^q)$ on the space of coherent states $k_w$, $w\in \comp^n$, given by (\ref{eq2.24}). Let us first make the following well known observation, see~\cite{BaCoIs}.

\begin{lemma}
\label{weak_limit}
We have $k_w\rightarrow 0$ weakly in $H_{\Phi_0}(\comp^n)$, as $\abs{w}\rightarrow \infty$.
\end{lemma}
\begin{proof}
We have for some constant $C\neq 0$,
\begeq
\label{eq3.4}
(k_w,k_z)_{H_{\Phi_0}({\bf C}^n)} = C\, e^{2\Psi_0(z,\overline{w}) - \Phi_0(z) - \Phi_0(w)},\quad w,z\in \comp^n,
\endeq
and therefore, in view of (\ref{eq2.23}), we obtain that $(k_w,k_z)_{H_{\Phi_0}({\bf C}^n)}\rightarrow 0$ as $\abs{w}\rightarrow \infty$. Taking linear combinations of the $k_z$'s we get
\begeq
\label{eq3.5}
(k_w,g)_{H_{\Phi_0}({\bf C}^n)} \rightarrow 0,
\endeq
as $\abs{w}\rightarrow \infty$, for all $g$ in a dense subspace of $H_{\Phi_0}(\comp^n)$, and this implies the result.
\end{proof}

\medskip
\noindent
Lemma \ref{weak_limit} and the compactness of ${\rm Top}(e^q)$ shows that ${\rm Top}(e^q)k_w \rightarrow 0$ in $H_{\Phi_0}(\comp^n)$, as $\abs{w}\rightarrow \infty$, and using (\ref{eq2.43}), (\ref{eq2.41}) we obtain that
\begeq
\label{eq3.6}
2{\rm Re}\, f(x,\overline{w}) < \Phi_0(x) + \Phi_0(w), \quad (0,0)\neq (x,w) \in \comp^n_x \times \comp^n_w.
\endeq
The strict positivity of the canonical transformation $\kappa$ in (\ref{eq2.5}) relative to $\Lambda_{\Phi_0}$ follows now from (\ref{eq3.6}), Proposition \ref{prop_positive}, and a straightforward modification of~\cite[Proposition 3.2]{CoHiSj}. The proof of Theorem \ref{theo_main2} is complete.

\medskip
\noindent
{\it Remark}. Let us observe that~\cite[Proposition 3.2]{CoHiSj} is only concerned with positive cano\-ni\-cal transformations, and the issue of strict positivity is not addressed there expli\-citly. What is being used in the discussion above is therefore a natural analogue of~\cite[Proposition 3.2]{CoHiSj} in the strictly positive case, allowing one to conclude that the canonical transformation $\widetilde{\kappa}$ in (\ref{eq2.47}) is strictly positive relative to $\Lambda_{\Phi_0}$, provided that the strict inequality (\ref{eq3.6}) holds. A proof of such an analogue of~\cite[Proposition 3.2]{CoHiSj} in the strictly positive case is obtained by inspecting the proof
of~\cite[Proposition 3.2]{CoHiSj}, making also use of the natural analogue of~\cite[Theorem 2.1]{CoHiSj} in the strictly positive case --- see also~\cite[Proposition 1.2.8]{HiSj15}, where this result, giving a characterization of strictly positive Lagrangian planes, is stated explicitly.

\section{An explicit example}
\label{sec_example}
\setcounter{equation}{0}
The purpose of this section is to discuss the boundedness and compactness properties for an explicit class of metaplectic Toeplitz operators on the Bargmann $H_{\Phi_0}(\comp^n)$, for a model weight $\Phi_0$, illustrating Theorem \ref{theo_main1} and Theorem \ref{theo_main2} in this case. It will be assumed throughout this section that
\begeq
\label{eq4.1}
\Phi_0(x) = \frac{\abs{x}^2}{4},\quad x\in \comp^n,
\endeq
so that the polarization is given by
\begeq
\label{eq4.1.1}
\Psi_0(x,y) = \frac{1}{4}x\cdot y,\quad x,y\in \comp^n.
\endeq

\medskip
\noindent
Let $\lambda \in \comp$, let $A$ be a complex symmetric $n\times n$ matrix, and let us set
\begeq
\label{eq4.2}
q(x) = \lambda\abs{x}^2 + A\overline{x}\cdot \overline{x},\quad x\in \comp^n.
\endeq
We shall assume that
\begeq
\label{eq4.3}
\displaystyle {\rm Re}\,\lambda + \norm{A} < \frac{1}{4},
\endeq
where $\norm{A}$ is the Euclidean operator norm of $A: \comp^n \rightarrow \comp^n$. It follows, in particular, that (\ref{eq1.8}) holds, so that the Toeplitz operator ${\rm Top}(e^q)$ is densely defined on $H_{\Phi_0}(\comp^n)$. The assumption (\ref{eq1.9}) is also satisfied, and our purpose here is to illustrate Theorem \ref{theo_main1} and Theorem \ref{theo_main2}, by characterizing the boundedness and compactness of ${\rm Top}(e^q)$ in terms of the parameters $\lambda$ and $\norm{A}$.

\begin{theo}
\label{theo_example}
Let $\displaystyle \Phi_0(x) = \frac{\abs{x}^2}{4}$, $x\in \comp^n$. Let $\lambda \in \comp$ and let $A$ be an $n \times n$ complex symmetric matrix such that $\displaystyle {\rm Re}\, \lambda + \norm{A} < \frac{1}{4}$. Let us set $q(x) = \lambda\abs{x}^2 + A\overline{x}\cdot \overline{x}$, $x\in \comp^n$. The Toeplitz operator
$$
{\rm Top}(e^q): H_{\Phi_0}(\comp^n) \rightarrow H_{\Phi_0}(\comp^n)
$$
is bounded if and only if
\begeq
\label{eq4.4}
4 \norm{A} \leq \frac{1 - \abs{\gamma}^2}{\abs{\gamma}^2},\quad \gamma = \frac{1}{1-2\lambda}.
\endeq
Furthermore, ${\rm Top}(e^q)$ is compact on $H_{\Phi_0}(\comp^n)$ precisely when the inequality in {\rm (\ref{eq4.4})} is strict.
\end{theo}
\begin{proof}
We shall first discuss the boundedness issue. It suffices, in view of Theorem \ref{theo_main1}, to show that the condition (\ref{eq4.4}) is satisfied precisely when the Weyl symbol of the operator ${\rm Top}(e^q)$ is bounded along $\Lambda_{\Phi_0}$. Here, rather than computing the Weyl symbol of ${\rm Top}(e^{q})$ by evaluating a suitable Gaussian integral, cf. (\ref{eq1.7}), it will be more convenient to show directly that the complex linear canonical transformation
\begeq
\label{eq4.5}
\kappa: \left(y,-F'_y(x,y,\theta)\right) \mapsto \left(x,F'_x(x,y,\theta)\right),\quad F'_{\theta}(x,y,\theta) = 0,
\endeq
where
\begeq
\label{eq4.6}
F(x,y,\theta) = \frac{2}{i}\left(\Psi_0(x,\theta) - \Psi_0(y,\theta)\right) + \frac{1}{i}Q(y,\theta) = \frac{1}{i}\left(\frac{(x-y)\cdot\theta}{2} + \lambda y\cdot \theta + A\theta\cdot \theta\right),
\endeq
is positive relative to $\Lambda_{\Phi_0}$.

\medskip
\noindent
The critical manifold of the non-degenerate phase function $F(x,y,\theta)$ in (\ref{eq4.6}) is given by $F'_{\theta}(x,y,\theta) = 0 \Longleftrightarrow x = (1-2\lambda)y - 4A\theta$, and a simple computation using (\ref{eq4.5}), (\ref{eq4.6}) shows that the canonical transformation $\kappa$ is given by
\begeq
\label{eq4.7}
\kappa: \comp^{2n} \ni (y,\eta) \mapsto \left((1-2\lambda)y - \frac{8iA\eta}{1-2\lambda},\frac{\eta}{1-2\lambda}\right) = \left(\frac{y}{\gamma} - 8i\gamma A\eta, \gamma \eta\right) \in \comp^{2n}.
\endeq
It follows from~\cite[equation (2.4)]{CoHiSj} that the anti-linear involution $\iota_{\Phi_0}: \comp^{2n} \rightarrow \comp^{2n}$ fixing $\Lambda_{\Phi_0}$, is given by
\begeq
\label{eq4.8}
\iota_{\Phi_0}: (y,\eta) \mapsto \left(\frac{2\overline{\eta}}{i}, {\frac{\overline{y}}{2i}}\right),
\endeq
and we have therefore,
\begeq
\label{eq4.9}
\frac{1}{i} \sigma\left((y,\eta),\iota_{\Phi_0}(y,\eta)\right) = \frac{1}{i} \sigma\left((y,\eta), \left(\frac{2\overline{\eta}}{i}, {\frac{\overline{y}}{2i}}\right)\right) = \frac{1}{2}\abs{y}^2 - 2\abs{\eta}^2.
\endeq
Recalling (\ref{eq2.6}), we conclude that the canonical transformation $\kappa$ in (\ref{eq4.5}) is positive relative to $\Lambda_{\Phi_0}$ precisely when we have
\begeq
\label{eq4.10}
\abs{\frac{y}{\gamma} - 8i\gamma A\eta}^2 - \abs{y}^2 + 4\left(1-\abs{\gamma}^2\right)\abs{\eta}^2 \geq 0,\quad (y,\eta)\in \comp^{2n},
\endeq
or in other words,
\begeq
\label{eq4.11}
\left(\frac{1 - \abs{\gamma}^2}{\abs{\gamma}^2}\right)\abs{y}^2 + 16 {\rm Re}\, \left(i\frac{\overline{\gamma}}{\gamma} \overline{A}\overline{\eta}\cdot y\right) + 64 \abs{\gamma}^2\abs{A\eta}^2 + 4\left(1-\abs{\gamma}^2\right)\abs{\eta}^2 \geq 0,\quad (y,\eta)\in \comp^{2n}.
\endeq
It is now elementary to check that the positivity property (\ref{eq4.11}) is implied by (\ref{eq4.4}), and when doing so we may assume that the matrix $A$ is non-vanishing, so that $\abs{\gamma} < 1$. The property (\ref{eq4.11}) may therefore be equivalently rewritten as follows,
\begeq
\label{eq4.12}
\abs{y}^2 + 2{\rm Re}\, \left(i\overline{a}\, \overline{A}\overline{\eta}\cdot y\right) + b\abs{A\eta}^2 + c \abs{\eta}^2 \geq 0, \quad (y,\eta)\in \comp^{2n},
\endeq
or in other words,
\begeq
\label{eq4.13}
\abs{y - ia A\eta}^2 + \left(b - \abs{a}^2\right) \abs{A\eta}^2 + c\abs{\eta}^2 \geq 0, \quad (y,\eta)\in \comp^{2n}.
\endeq
Here
\begeq
\label{eq4.14}
a = \frac{8 \abs{\gamma}^2}{1-\abs{\gamma}^2} \frac{\gamma}{\overline{\gamma}}, \quad b = \frac{64 \abs{\gamma}^4}{1-\abs{\gamma}^2}>0, \quad
c = 4\abs{\gamma}^2>0.
\endeq
Now (\ref{eq4.13}) holds precisely when
\begeq
\label{eq4.15}
c \abs{\eta}^2 \geq \left(\abs{a}^2 - b\right)\abs{A\eta}^2, \quad \eta \in \comp^n,
\endeq
and observing that, in view of (\ref{eq4.14}),
$$
\abs{a}^2 - b = \frac{64 \abs{\gamma}^6}{\left(1-\abs{\gamma}^2\right)^2},
$$
we immediately conclude that (\ref{eq4.15}) follows from (\ref{eq4.4}). Similar arguments show that the condition (\ref{eq4.4}) is also necessary for the positivity of the canonical transformation $\kappa$ in (\ref{eq4.5}) relative to $\Lambda_{\Phi_0}$, so that (\ref{eq4.4}) holds precisely when the Weyl symbol of the operator ${\rm Top}(e^{q})$ is bounded. The compactness of ${\rm Top}(e^q)$ can be characterized in a similar way, as it is equivalent to the strict positivity of the canonical transformation $\kappa$ in (\ref{eq4.5}), relative to $\Lambda_{\Phi_0}$.
\end{proof}

\end{document}